\documentclass{article} 

 \usepackage[usenames,dvipsnames]{pstricks}
 \usepackage{epsfig}
 \usepackage{pst-grad} 
 \usepackage{pst-plot} 
 \usepackage[space]{grffile} 
 \usepackage{etoolbox} 
 \makeatletter 
 \patchcmd\Gread@eps{\@inputcheck#1 }{\@inputcheck"#1"\relax}{}{}
 \makeatother

\usepackage[usenames,dvipsnames]{pstricks}
\usepackage{epsfig}
\usepackage{pst-grad} 
\usepackage{pst-plot} 
\usepackage[space]{grffile} 
\usepackage{etoolbox} 
\makeatletter 
\patchcmd\Gread@eps{\@inputcheck#1 }{\@inputcheck"#1"\relax}{}{}
\makeatother

 \usepackage[usenames,dvipsnames]{pstricks}
 \usepackage{epsfig}
 \usepackage{pst-grad} 
 \usepackage{pst-plot} 
 \usepackage[space]{grffile} 
 \usepackage{etoolbox} 
 \makeatletter 
 \patchcmd\Gread@eps{\@inputcheck#1 }{\@inputcheck"#1"\relax}{}{}
 \makeatother

\usepackage{amsmath}
\usepackage{amssymb}
\usepackage{amsthm}
\usepackage[latin1]{inputenc}
     
\usepackage{graphicx}

\usepackage[usenames,dvipsnames]{pstricks}
\usepackage{epsfig}
\usepackage{pst-grad} 
\usepackage{pst-plot} 

\usepackage{ifthen}
\usepackage{color}

\newcommand{\intav}[1]{\mathchoice {\mathop{\vrule width 6pt height 3 pt depth  -2.5pt
\kern -8pt \intop}\nolimits_{\kern -6pt#1}} {\mathop{\vrule width
5pt height 3  pt depth -2.6pt \kern -6pt \intop}\nolimits_{#1}}
{\mathop{\vrule width 5pt height 3 pt depth -2.6pt \kern -6pt
\intop}\nolimits_{#1}} {\mathop{\vrule width 5pt height 3 pt depth
-2.6pt \kern -6pt \intop}\nolimits_{#1}}}

\def\polhk#1{\setbox0=\hbox{#1}{\ooalign{\hidewidth\lower1.5ex\hbox{`}\hidewidth\crcr\unhbox0}}}

\newcommand{\tr}{\operatorname{Tr}}

\renewcommand{\dim}{\operatorname{dim}}
\newcommand{\llip}{\operatorname{Log-Lip}}

\newtheorem{teo}{Theorem}

\newtheorem{Proposition}{Proposition}
\newtheorem{Remark}{Remark}
\newtheorem{Assumption}{A}

\begin{document}

\title{A note on the density of the partial regularity result in the class of viscosity solutions}
\author{Disson dos Prazeres, Edgard A. Pimentel and Giane C. Rampasso}

\date{\today} 

\maketitle

\begin{abstract}

\noindent We establish the density of the partial regularity result in the class of continuous viscosity solutions. Given a fully nonlinear equation, we prove the existence of a sequence entitled to the partial regularity result, approximating its solutions. Distinct conditions on the operator driving the equation lead to density in different topologies. Our findings include applications to inhomogeneous problems, with variable-coefficients models.
\medskip


\medskip 

\end{abstract}

\vspace{.1in}

\section{Introduction}\label{sec_introduction}
In this note we examine the density of the \emph{partial regularity result} in the class $S(\lambda,\Lambda)$ of continuous viscosity solutions to 
\begin{equation}\label{eq_main}
	F(D^2u)\,=\,0\;\;\;\;\;\mbox{in}\;\;\;\;\;B_1,
\end{equation}
where $F:\mathcal{S}(d)\to\mathbb{R}$ is merely a $(\lambda,\Lambda)$-elliptic operator.

Given a viscosity solution $u\in\mathcal{C}(B_1)$ to \eqref{eq_main}, we prove the existence of a sequence $(u_n)_{n\in\mathbb{N}}$, for which the partial regularity result is available, converging to $u$. Different modes of convergence are examined, under distinct assumptions on $F$. We also consider the case of inhomogeneous problems, in the presence of variable coefficients. Our main contribution in this note is to implement a mollification strategy for $F$. Our analysis is motivated by the class of results put forward in \cite{CafSil10}.

The regularity theory for fully nonlinear elliptic problems occupies a prominent role in the analysis of partial differential equations (PDE).  An important development in this realm stems from the works of N. Krylov and M. Safonov, and Trudinger; see \cite{Krylov-Safonov_1980, Trudinger_1989}; see also \cite{Caf89} and \cite[Chapter 8.2]{CafCab95}. It ensures that if $u$ is a viscosity solution to \eqref{eq_main}, then $u\in\mathcal{C}^{1,\alpha_0}_{loc}(B_1)$, for some $\alpha_0\in(0,1)$ unknown. In addition, the appropriate estimates are available. 
This result is also understood as a fully nonlinear counterpart of the De Giorgi-Nash-Moser theory.

If ellipticity is combined with convexity of the operator $F$, further levels of regularity are unveiled. This is the content of the Evans-Krylov theory \cite{Evans_1982,Krylov_1982}. Under these assumptions, solutions to \eqref{eq_main} are actually of class $\mathcal{C}^{2,\alpha}$, locally, with estimates. Once again, $\alpha\in(0,1)$ is unknown. Independently developed by L.C. Evans and N. Krylov, this corpus of results unlocks a set of conditions under which \emph{classical solutions} are available.

In \cite{Caf89}, L. Caffarelli examines viscosity solutions to 
\begin{equation}\label{eq_main2}
	F(D^2u,x)\,=\,f\;\;\;\;\;\mbox{in}\;\;\;\;\;B_1.
\end{equation}
The fundamental idea in that paper is to relate \eqref{eq_main2} with the homogeneous equations driven by the fixed-coefficients counterpart of $F(M,x)$. Under natural assumptions on the oscillation of the operator and the integrability of the source term, the author develops a regularity theory in H\"older and Sobolev spaces. See \cite{CafCab95}. For related developments, see also \cite{Escauriaza_1993,Swiech_1997,Winter_2009}.

The corpus of results comprised by the $\mathcal{C}^{1,\alpha_0}$-estimates, and the Evans-Krylov and Caffarelli regularity theories suggests a fundamental question.  It regards the highest regularity level expected to hold in the presence of uniform ellipticity alone. In other words, one asks if the $\mathcal{C}^{1,\alpha_0}$-regularity theory is optimal in the absence of further conditions on the operator.

This question was set in the affirmative only recently. In \cite{Nadirashvili-Vladut_2007,Nadirashvili-Vladut_2008,Nadirashvili-Vladut_2011a,Nadirashvili-Vladut_2011b}, N. Nadirashvili and S. Vl\u{a}du\c{t} produced a number of important counterexamples. First, the authors obtained a $(\lambda,\Lambda)$-elliptic equation with solutions whose Hessian matrices blow-up. Moreover, for every $\tau\in(0,1)$, the authors managed to design a $(\lambda,\Lambda)$-elliptic operator $F_\tau$ such that solutions to 
\[
	F_\tau(D^2v)\,=\,0\;\;\;\;\;\mbox{in}\;\;\;\;\;B_1
\]	
lack $\mathcal{C}^{1,\tau}$-regularity. In brief, the former examples imply that ellipticity is not enough to enforce $\mathcal{C}^{1,1}$-regularity. Also, the conjecture that \emph{solutions are of class $\mathcal{C}^{1,\alpha}$ for every $\alpha\in(0,1)$}  is proven false. Surprisingly enough, the Krylov-Safonov regularity theory is optimal if the operator driving the equation is no more than elliptic.

This context motivates the development of methods to unlock further regularity of the solutions. In line with this effort, one finds a number of contributions, as several authors have been working on a variety of directions. We choose to mention two among those.

In \cite{Savin_2007} O. Savin considers an equation of the form
\begin{equation}\label{eq_savin}
	F(D^2u,Du,u,x)\,=\,0\;\;\;\;\;\mbox{in}\;\;\;\;\;B_1
\end{equation}
and impose a few conditions on $F$. In that paper, the operator is supposed to be degenerate elliptic. Also, $F$ is uniformly elliptic in a neighborhood of the origin and verifies $F(0,0,0,x)=0$. In addition, the author assumes $F$ to be twice differentiable, with a modulus of continuity for the Hessian. Under those conditions, there exists a constant $\sigma_0>0$ such that, if a viscosity solution to \eqref{eq_savin} satisfies 
\[
	\left\|u\right\|_{L^\infty(B_1)}\,\leq\,\sigma_0
\]
then $u\in\mathcal{C}^{2,\alpha}_{loc}(B_1)$, with estimates. The $\sigma_0$-smallness condition imposed on the $L^\infty$-norm of the solutions in known as \emph{flatness}. The associated statement is referred to as \emph{flatness implies $\mathcal{C}^{2,\alpha}$}. Of particular interest in \cite{Savin_2007} is the condition on the differentiability of the operator with respect to its Hessian entry.

In \cite{Armstrong-Silvestre-Smart_2012} S. Armstrong, L. Silvestre and C. Smart examine viscosity solutions to 
\[
	F(D^2u)\,=\,0\;\;\;\;\;\mbox{in}\;\;\;\;\;B_1
\]
under differentiability assumptions on $F$. In fact, the authors suppose $F\in\mathcal{C}^1(\mathcal{S}(d))$ with a modulus of continuity for the gradient. Were those conditions in force, solutions would be of class $\mathcal{C}^{2,\alpha}$ in $B_1\setminus\Omega$, with the Hausdorff dimension of $\Omega$ strictly below $d$. If $d=2$, the set where $\mathcal{C}^{2,\alpha}$-regularity fails would be at most a line. 

This advance has been known as \emph{partial regularity result} and represents an important advance in the theory. It addresses nonconvex operators by imposing a (uniform) differentiability condition. The authors resort to some aspects of \cite{Savin_2007} and the measure control for the Hessian put forward in \cite{Lin_1986}.  As a by-product of their analysis, a conjecture on the precise formula of the Fanghua Lin's exponent is stated. We refer the reader to \cite{Lin_1986}; see also \cite[Proposition 7.4]{CafCab95}.

Although stated in fairly general terms, the partial regularity result does not include important toy-models of the literature. An example is the Isaacs equation. For $A:\mathcal{A}\times\mathcal{B}\to\mathcal{S}(d)$, mapping $\mathcal{A}\times\mathcal{B}\ni (\alpha,\beta)\mapsto A_{\alpha,\beta}$, we consider 
\begin{equation}\label{eq_isaacs}
	\sup_{\beta\in\mathcal{B}}\,\inf_{\alpha\in\mathcal{A}}\,\left(-\tr\left(A_{\alpha,\beta}D^2u\right)\right)\,=\,0\;\;\;\;\;\mbox{in}\;\;\;\;\;B_1.
\end{equation}

The Isaacs equation is of utmost relevance in the theory of fully nonlinear elliptic equations. This is partly because a fully nonlinear problem can be formulated in terms of such an equation, for appropriate families of operators; see, for instance, \cite{Cabre-Caffarelli_2003}. In addition, solutions to \eqref{eq_isaacs} are value functions of stochastic two-players, zero-sum, differential games; we refer the reader to \cite{Barron-Evans-Jensen_1984,Evans-Souganidis_1984,Fleming-Souganidis_1989}.

In the present work we produce approximation results for fully nonlinear $(\lambda,\Lambda)$-elliptic operators lacking differentiability. We relate their solutions to sequences $(u_n)_{n\in\mathbb{N}}$ in the class of viscosity solutions $S(\lambda,\Lambda)$. 
For similar results involving nonlocal operators and smooth approximations, see the work of L. Caffarelli and L. Silvestre \cite{CafSil10}. For the density of improved regularity in the class of viscosity solutions, we refer the reader to \cite{PimTei16} and \cite{PimSan18}.

Here we are interested in the following class of results. Suppose $F$ is a $(\lambda,\Lambda)$-elliptic operator and let $u\in\mathcal{C}(B_1)$ be a viscosity solution to \eqref{eq_main}. Although it is not possible to extend the partial regularity result to $u$, we prove the existence of a sequence $(u_n)_{n\in\mathbb{N}}$ converging to $u$ and for which $\mathcal{C}^{2,\alpha}$-regularity is available, except in a subset $\Omega\subset B_1$ of Hausdorff measure strictly less than $d$. Different assumptions on the operator $F$ yield distinct modes of convergence. Ultimately, if one has an interest in properties closed under certain limits, the starting-point of the theory can be, in general, the partial regularity result. Our first main theorem reads as follows:

\begin{teo}[Density in the $\mathcal{C}^{1,\beta}$-topology]\label{teo_main1}
Let $u\in\mathcal{C}(B_1)$ be a viscosity solution to \eqref{eq_main}. Suppose A\ref{assump_Felliptic}, to be detailed further, is in effect. For every $\alpha\in(0,1)$, there exists $\Omega\subset B_1$, a constant $\delta>0$ and a sequence $(u_n)_{n\in\mathbb{N}}$ satisfying
\begin{enumerate}
\item $u_n\in S(\lambda,\Lambda)\cap\mathcal{C}^{2,\alpha}(B_1\setminus \Omega)$, for every $n\in\mathbb{N}$;
\item $u_n\longrightarrow u$ in the $\mathcal{C}^{1,\beta}$-topology, for every $\beta\in(0,\beta_0)$, where $\beta_0\in(0,1)$ is the exponent driving the $\mathcal{C}^{1,\beta_0}$-regularity for $F=0$, and
\item $\dim_\mathcal{H}\Omega\,<\,d\,-\,\delta$.
\end{enumerate}
\end{teo}

The former theorem states that, although $F$ is not differentiable, $u$ can be approximated in the $\mathcal{C}^{1,\beta}$-topology by functions under the scope of the partial regularity result, in the same viscosity class. Therefore, when studying properties closed under $\mathcal{C}^{1,\beta}$-limits, one can suppose the partial regularity result is available in general. Under further conditions on the operator $F$ the convergence of the approximating sequence takes place in a finer topology.

Indeed, if we suppose that $F$ has a recession profile $F^*$ with $\mathcal{C}^{1,1}$-estimates, it is possible to prove that $u_n\to u$ in the $\mathcal{C}^{1,\llip}$-topology. This occurs whenever $F^*$ is convex/concave. For details on the notion of recession function, we refer the reader to \cite{PimSan18,PimTei16,SilTei15}. Our second result is the following.

\begin{teo}[Density in the $\mathcal{C}^{1,\llip}$-topology]\label{teo_main2}
Let $u\in\mathcal{C}(B_1)$ be a viscosity solution to \eqref{eq_main}. Suppose A\ref{assump_Felliptic}-A\ref{assump_recconvex}, to be detailed further, are in force. For every $\alpha\in(0,1)$, there exists $\Omega\subset B_1$, a constant $\delta>0$ and a sequence $(u_n)_{n\in\mathbb{N}}$ satisfying
\begin{enumerate}
\item $u_n\in S(\lambda,\Lambda)\cap\mathcal{C}^{2,\alpha}(B_1\setminus \Omega)$, for every $n\in\mathbb{N}$;
\item $u_n\longrightarrow u$ in the $\mathcal{C}^{1,\llip}$-topology, and
\item $\dim_\mathcal{H}\Omega\,<\,d\,-\,\delta$.
\end{enumerate}
\end{teo}

This result refines Theorem \ref{teo_main1} as it ensures the convergence of an approximating sequence in the $\mathcal{C}^{1,\llip}$-topology. In particular, it includes the case of the $\mathcal{C}^{1,\beta}$-topology for every $\beta\in(0,1)$. 

Our third result covers the case of inhomogeneous equations, governed by operators with variable coefficients. In \cite{dosTei16} the authors examine conditions on the structure of \eqref{eq_main2} under which flatness implies improved regularity of the solutions. Their findings are sharp, in the sense that Lipschitz continuity of the data with respect to $x$ leads to $\mathcal{C}^{2,\alpha}$-regularity. In contrast, mere continuity implies estimates in $\mathcal{C}^{1,\llip}$-spaces. Consequential on the findings in that paper is a partial regularity result for the solutions to \eqref{eq_main2} provided $F$ is well-prepared; see Proposition \ref{prop_prrvc}. We refer the reader to \cite[Section 5]{dosTei16}.

Our techniques build upon \cite[Section 5]{dosTei16} and produce an approximation result for the solutions to \eqref{eq_main2}. This is reported in our third main result. 

\begin{teo}[Inhomogeneous, variable-coefficients equations]\label{teo_main3}
Let $u\in\mathcal{C}(B_1)$ be a viscosity solution to \eqref{eq_main2}. Suppose A\ref{assump_Felliptic}, A\ref{assump_Fosc} and A\ref{assump_source}, to be detailed further, are in force. For every $\alpha\in(0,1)$, there exists $\Omega\subset B_1$, a constant $\delta>0$ and a sequence $(u_n)_{n\in\mathbb{N}}$ satisfying
\begin{enumerate}
\item $u_n\in S(\lambda,\Lambda)\cap\mathcal{C}^{2,\alpha}(B_1\setminus \Omega)$, for every $n\in\mathbb{N}$;
\item $u_n\longrightarrow u$ in the $\mathcal{C}^{1,\beta}$-topology, for every $\beta\in(0,\beta_{KS})$, where $\beta_{KS}$ is defined as before, and
\item $\dim_\mathcal{H}\Omega\,<\,d\,-\,\delta$.
\end{enumerate}
\end{teo}

We argue through a regularization of the operator $F$, by means of a molification strategy. That is, for $\varepsilon>0$ we introduce the auxiliary operator
\[
	F_\varepsilon(M)\,:=\,\left(F\,\ast\,\eta_\varepsilon\right)(M),
\]
where $\eta_\varepsilon:\mathbb{R}^\frac{d(d+1)}{2}\to\mathbb{R}$ is a symmetric molifying kernel. It is evident that $F_\varepsilon$ is continuously differentiable; therefore, solutions to
\[
	F_\varepsilon(D^2u_\varepsilon)\,=\,0\;\;\;\;\;\mbox{in}\;\;\;\;\;B_1
\]
are under the scope of the partial regularity result. In addition, $F_\varepsilon$ converges locally uniformly to the original operator $F$, since the latter is $(\lambda,\Lambda)$-elliptic. 

The technical difficulties are in verifying that pivotal properties of the original operator are preserved under convolution. Those include ellipticity, oscillation control and regularity profiles associated with the recession function.

Once those difficulties are tackled, the family $(u_\varepsilon)_{\varepsilon>0}$ is entitled to regularity results which ensure convergence in the desired topologies.

The remainder of this paper is organized as follows: Section \ref{subsec_ma} details the main assumptions under which we work. In Section \ref{subsec_aux} we gather a few preliminary results, whereas Section \ref{sec_potc} presents properties of regularized operators. Section \ref{sec_proof12} reports the proofs of Theorem \ref{teo_main1} and Theorem \ref{teo_main2}. The proof of Theorem \ref{teo_main3} is the subject of Section \ref{sec_generalizations}.

\section{Main assumptions and preliminary material}\label{sec_mapm}

In the sequel, we detail our main assumptions. Moreover, we gather preliminary results to which we resort throughout the paper.

\subsection{Main assumptions}\label{subsec_ma}

We start by detailing the hypotheses used in this article. With $B_1$ we denote the open unit ball in the Euclidean space $\mathbb{R}^d$. Also, the space of $d\times d$ symmetric matrices is denoted by $\mathcal{S}(d).$ Our first assumption concerns the uniform ellipticity of $F$.

\begin{Assumption}[Uniform ellipticity]\label{assump_Felliptic}
We suppose $F:\mathcal{S}(d)\to\mathbb{R}$ is $(\lambda,\Lambda)$-elliptic. That is, for $0<\lambda\leq\Lambda$, we have
\[
	\lambda\left\|N\right\|\,\leq\,F(M\,+\,N)\,-\,F(M)\,\leq\,\Lambda\left\|N\right\|,
\]
for every $M,\,N\in\mathcal{S}(d)$, with $N\geq 0$. In addition, we suppose $F(0)=0$.
\end{Assumption}

It is a known fact that  A\ref{assump_Felliptic} implies that $F$ is uniformly Lipschitz, with constant $K_F=K_F(\lambda,\Lambda,d)$. In addition, we notice that requiring $F(0)=0$ imposes no further restrictions on the operator, since $G:=F-F(0)$ has the same ellipticity constants as $F$. In the case of variable coefficients $F(M,x)$, A\ref{assump_Felliptic} adjusts as usual.

The next assumption concerns the recession profile associated with the operator $F$. We recall that
\[
 F^*(M)\,:=\,\lim_{\mu\to 0}\mu F(\mu^{-1}M).
\]

\begin{Assumption}[Convexity of the recession profile]\label{assump_recconvex}
We suppose the recession profile $F^*$ to be convex.
\end{Assumption}

The recession profile associated with a given operator is an important tool in the study of regularity issues. For its main properties, and usual hypotheses placed on this object, we refer the reader to \cite{PimSan18,SilTei15}.

Our arguments rely on the convolution of mappings acting on symmetric matrices. For ease of presentation we set
\[
	d^*\,:=\,\frac{d(d\,+\,1)}{2}.
\]
Notice that the algebraic dimension on $\mathcal{S}(d)$ is precisely $d^*$.

In Section \ref{sec_generalizations} we address fully nonlinear operators $F=F(M,x)$, with explicit dependence on the space variable $x\in B_1$, governing inhomogeneous equations. The next assumptions set the conditions under which we operate in this setting.


\begin{Assumption}[Oscillation control]\label{assump_Fosc}
We suppose the oscillation measure
\[
	\beta(x,x_0)\,:=\,\sup_{M\in\mathcal{S}(d)}\,\frac{\left|F(M,x)\,-\,F(M,x_0)\right|}{\left\|M\right\|\,+\,1}
\]
satisfies
\[
	\int_{B_r(x_0)}\beta(x,x_0)^d{\bf d}x\,\leq\,\beta_0^dr^d,
\]
for some $\beta_0>0$.
\end{Assumption}

Introduced in \cite{Caf89}, the former assumption is a cornerstone of the theory for variable-coefficients operators. See \cite{CafCab95} for variants of A\ref{assump_Fosc} appearing in distinct regularity regimes; see also \cite[Remark 8.2]{CafCab95} for the connection of $\beta(\cdot,x_0)$ with the H\"older-regularity of $F$ with respect to $x$. We conclude with an assumption on the source term $f$.

\begin{Assumption}[Inhomogeneous setting]\label{assump_source}We suppose $f\in L^\infty(B_1)$ to be a continuous function.
\end{Assumption}

The former assumption is instrumental in framing our analysis in the context of continuous viscosity solutions. In the next section we collect a few preliminary results.

\subsection{A few preliminary results}\label{subsec_aux}

In what follows we collect former results and developments used in the paper. We begin with the partial regularity result.

\begin{Proposition}[Partial regularity result]\label{prop_prr}
Let $u\in\mathcal{C}(B_1)$ be a viscosity solution to \eqref{eq_main}. Suppose A\ref{assump_Felliptic} is in force. Suppose further $F\in\mathcal{C}^1(\mathcal{S}(d))$, with a uniform modulus of continuity. For every $\alpha\in(0,1)$ there exists $\Omega\subset B_1$ and a universal constant $\delta>0$ such that $u\in\mathcal{C}^{2,\alpha}(B_1\setminus\Omega)$ and $\dim_\mathcal{H}\Omega<d-\delta$.
\end{Proposition}

For the proof of Proposition \ref{prop_prr} we refer the reader to \cite{Armstrong-Silvestre-Smart_2012}. This result states that solutions to nonlinear equations driven by regular operators fail $\mathcal{C}^{2,\alpha}$-regularity only within a set of controlled Hausdorff dimension. We also resort to a generalization of Proposition \ref{prop_prr} to the context of inhomogeneous, variable-coefficients equations; see \cite[Corollary 5.2]{dosTei16}.

\begin{Proposition}\label{prop_prrvc}
Let $u\in\mathcal{C}(B_1)$ be a viscosity solution to \eqref{eq_main2}. Suppose A\ref{assump_Felliptic} is in force. For every $x\in B_1$, suppose $F(\,\cdot\,,x)\in\mathcal{C}^1(\mathcal{S}(d))$, with a modulus of continuity not depending on $x$. In addition, assume $F$ and the source term $f$ are Lipschitz continuous with respect to $x$. Then, for every $\alpha\in(0,1)$ there exists $\Omega\subset B_1$ and a universal constant $\delta>0$ such that $u\in\mathcal{C}^{2,\alpha}(B_1\setminus\Omega)$ and $\dim_\mathcal{H}\Omega<d-\delta$.
\end{Proposition}

For the proof of the former proposition we refer the reader to \cite{dosTei16}. In fact, in that paper the authors establish a variant of this result, in the context of inhomogeneous fixed-coefficients equations. However, the Lipschitz-continuity of $F$ with respect to $x$ unlocks the argument leading to Proposition \ref{prop_prrvc}.

Our analysis also depends on former regularity results. We are interested in improved regularity in borderline-H\"older spaces.



\begin{Proposition}[Regularity in $\mathcal{C}^{1,\llip}$-spaces]\label{prop_c1loglip}
Let $u\in\mathcal{C}(B_1)$ be a viscosity solution to
\[
	F(D^2u)\,=\,f\;\;\;\;\;\,\mbox{in}\;\;\;\;\;B_1,
\]
where $f\in L^\infty(B_1)$. Suppose A\ref{assump_Felliptic} and A\ref{assump_recconvex} are in force. Then $u\in \mathcal{C}^{1,\llip}_{loc}(B_1)$ and there exists $C>0$, depending solely on the dimension $d$ and the ellipticity constants $\lambda$ and $\Lambda$, such that 
\[
	\sup_{x\in B_r(x_0)}\left|u(x)\,-\,u(x_0)\,-\,Du(x_0)\cdot(x\,-\,x_0)\right|\,\leq\,Cr^2\ln r^{-1},
\]
for every $B_r(x_0)\Subset B_1$.
\end{Proposition}

We refer the reader to \cite{SilTei15} for a proof of Proposition \ref{prop_c1loglip}. We notice that under an appropriate oscillation control Proposition \ref{prop_c1loglip} extends to the case of variable coefficients. Also, the requirement $f\in L^\infty(B_1)$ in the latter can be replaced with $f\in \mbox{BMO}(B_1)$. See \cite{SilTei15}; see also \cite{PimSan18}. 

In the sequel we examine the effect of convolutions on the properties of fully nonlinear elliptic operators.

\section{Properties of the convolution}\label{sec_potc}

Throughout this section we consider standard mollifiers $\eta_\varepsilon\in\mathcal{C}^\infty(\mathbb{R}^{d^*})$, for $\varepsilon>0$, and define
\begin{equation}\label{eq_conv}
	F_\varepsilon(M)\,:=\,\left(F\,\star\,\eta_\varepsilon\right)(M),
\end{equation}
for every $M\in\mathcal{S}(d)$. 

Under A\ref{assump_Felliptic}, it is straightforward to notice the convolution in \eqref{eq_conv} is well defined. Moreover, that $F_\varepsilon\to F$ uniformly on compact subsets of $\mathcal{S}(d)$, as $\varepsilon\to 0$. Finally, $F_\varepsilon\in\mathcal{C}^\infty(\mathcal{S}(d))$. See \cite[Theorems C.19 and C.20]{Leo09}.

We move forward with a proposition.

\begin{Proposition}\label{prop_elliptic}
Suppose A\ref{assump_Felliptic} holds true. Then, the operator $F_\varepsilon$ is $(\lambda,\Lambda)$-elliptic for every $\varepsilon>0$.
\end{Proposition}
\begin{proof}
We have
\[
	F_\varepsilon(M\,+\,N)\,-\,F_\varepsilon(M)\,=\,\int_{\mathbb{R}^{d^*}}\left[F(M+N-Q)-F(M-Q)\right]\eta_\varepsilon(Q){\bf d}Q.
\]
Hence, 
\begin{equation}\label{eq_ellip1}
	\begin{split}
		\left|F_\varepsilon(M\,+\,N)\,-\,F_\varepsilon(M)\right|\,&\leq\,\int_{\mathbb{R}^{d^*}}\left|F(M+N-Q)-F(M-Q)\right|\eta_\varepsilon(Q){\bf d}Q\\
			&\leq\,\Lambda\left\|N\right\|\int_{\mathbb{R}^{d^*}}\eta_\varepsilon(Q){\bf d}Q\\
			&=\,\Lambda\left\|N\right\|.
	\end{split}
\end{equation}
An analogous computation yields
\begin{equation}\label{eq_ellip2}
	\left|F_\varepsilon(M\,+\,N)\,-\,F_\varepsilon(M)\right|\,\geq \lambda\left\|N\right\|.
\end{equation}
By combining \eqref{eq_ellip1} and \eqref{eq_ellip2} we complete the proof.
\end{proof}

Therefore, the family $(F_\varepsilon)_{\varepsilon>0}$ inherits the uniform ellipticity of the original operator $F$. However, the mollification also preserves asymptotic properties. In fact, if we suppose that $F^*$ is convex, then $(F_\varepsilon)^*$ is also convex, for every $\varepsilon>0$. As a consequence, by requiring $F^*$ to be convex, we ensure that $F_\varepsilon$ has a convex recession profile for every $\varepsilon>0$. 

Before the next proposition we observe that
\begin{equation}\label{op de recessao com Q}
	F^*(M)\,=\,\lim_{\mu\to 0}\mu F(\mu^{-1}M\,-\,Q)\;\;\;\mbox{for all}\;\;\;Q\,\in\,\mathcal{S}(d). 
\end{equation}
Indeed, notice that
\[
	\begin{split}
		\left|\mu F(\mu^{-1}M\,-\,Q)\,-\,\mu F(\mu^{-1}M)\right|\,\leq\,\mu K_{F}\left\|Q\right\|,
	\end{split}
\]
for every $Q\in\mathcal{S}(d)$.

\begin{Proposition}[Convexity of the regularized recession operator]\label{prop_recconvex}
Suppose A\ref{assump_Felliptic} and A\ref{assump_recconvex} are in force. Then, $(F_\varepsilon)^*$ is convex, for every $\varepsilon>0$.

\end{Proposition}
\begin{proof}
	Fix $0<\rho<1$ and notice that
	$$
(F_\varepsilon)^*(\rho M+(1-\rho)N)= \lim_{\mu\to 0}\mu \int_{\mathbb{R}^{d^*}} F(\mu^{-1}(\rho M+(1-\rho)N)-Q)\eta_\varepsilon(Q){\bf d}Q.
	$$
By \eqref{op de recessao com Q}, we obtain
$$
 (F_\varepsilon)^*(\rho M+(1-\rho)N)=\int_{\mathbb{R}^{d^*}} F^*((\rho M+(1-\rho)N))\eta_\varepsilon(Q){\bf d}Q.
$$
As A\ref{assump_recconvex} holds, we have
\[
	\begin{split}
		\int_{\mathbb{R}^{d^*}} F^*((\rho M+(1-\rho)N))\eta_\varepsilon(Q)dQ&\leq \rho\int_{\mathbb{R}^{d^*}} F^*( M)\eta_\varepsilon(Q){\bf d}Q\\
			&\quad+(1-\rho)\int_{\mathbb{R}^{d^*}} F^*(N)\eta_\varepsilon(Q){\bf d}Q.
	\end{split}
\]
Once again, by using \eqref{op de recessao com Q}, we have
$$
	\int_{\mathbb{R}^{d^*}} F^*( M)\eta_\varepsilon(Q){\bf d}Q=\lim_{\mu\to 0}\mu\int_{\mathbb{R}^{d^*}} F(\mu^{-1} M-Q)\eta_\varepsilon(Q){\bf d}Q=(F_\varepsilon)^*(M)
$$
and
$$
\int_{\mathbb{R}^{d^*}} F^*(N)\eta_\varepsilon(Q){\bf d}Q=\lim_{\mu\to 0}\mu\int_{\mathbb{R}^{d^*}} F((\mu^{-1}N-Q)\eta_\varepsilon(Q){\bf d}Q=(F_\varepsilon)^*(N).
$$
Hence,
$$
 (F_\varepsilon)^*(\rho M+(1-\rho)N)\leq \rho (F_\varepsilon)^*(M)+ (1-\rho)(F_\varepsilon)^*(N),
$$
which ends the proof.
\end{proof}

\begin{Remark}[Rate of convergence]\normalfont\label{rem_rate}
An important aspect concerning the recession profile of a given operator $F$ is the rate of convergence
\[
	\lim_{\mu\to 0}\mu F(\mu^{-1}M)\,=\,F^*(M).
\]
We observe that 
\[
	\begin{split}
		\left|\mu F_\varepsilon(\mu^{-1}M)-F_\varepsilon^*(M)\right|&=\int_{R^{d^*}}\left|\mu F(\mu^{-1}M-Q)-F^*(M-Q)\right|\eta_\varepsilon(Q){\bf d}Q\\
			&\leq\,o(1)\left(1\,+\,\left\|M\,-\,Q\right\|\right),
	\end{split}
\]
as $\mu\to 0$. In other words, the rate of convergence 
\[
	\mu F_\varepsilon(\mu^{-1}M)\,\longrightarrow\,F_\varepsilon^*(M)
\]
is independent of $\varepsilon>0$. Therefore, properties depending on the aforementioned rate are uniform along the family $(F_\varepsilon)_{\varepsilon>0}$.
\end{Remark}

In Section \ref{sec_generalizations} we examine \eqref{eq_main2}. In this case, the operator depends explicitly on the space variable. Here we argue through a double-convolution argument: let $F$ satisfy A\ref{assump_Felliptic} and A\ref{assump_Fosc}. We define
\[
	F_{\varepsilon,\tau}(M,x)\,:=\,\int_{B_1}\int_{\mathbb{R}^{d^*}}F(M\,-\,Q,x\,-\,y)\eta_{\varepsilon}(Q)\eta_\tau(y){\bf d}Q{\bf d}y.
\]

As before, $F_{\varepsilon,\tau}$ preserves the ellipticity of the original operator. However, when it comes to the regularity of the solutions to \eqref{eq_main2}, the oscillation
\[
	\beta_{\varepsilon,\tau}(x,x_0)\,:=\,\sup_{M\in\mathcal{S}(d)}\,\frac{|F_{\varepsilon,\tau}(M,x)\,-\,F_{\varepsilon,\tau}(M,x_0)|}{\left\|M\right\|\,+\,1}
\]
plays a relevant role. In fact, estimates in Sobolev and borderline-H\"older spaces require $\beta_{\varepsilon,\tau}(x,x_0)$ to satisfy
\[
	\int_{B_r(x_0)}\left|\beta_{\varepsilon,\tau}(x,x_0)\right|^d{\bf d}x\,\leq\,\beta_0^dr^d,
\]
for some $\beta_0>0$ and every $x_0\in B_1$. 

In what follows, we prove that $\beta_{\varepsilon,\tau}$ is controlled by the oscillation of the original operator $F(M,x)$, for small values of the parameters $\varepsilon$ and $\tau$.

\begin{Proposition}[Oscillation control]\label{prop_oscreg}
Suppose $F:\mathcal{S}(d)\times B_1\to\mathbb{R}$ satisfies A\ref{assump_Felliptic} and A\ref{assump_Fosc}. Then, there exist $\varepsilon^*>0$ and $\tau^*>0$ such that, if $\varepsilon<\varepsilon^*$ and $\tau<\tau^*$, we have
\[
	\left\|\beta_{\varepsilon,\tau}(\cdot,x_0)\right\|_{L^d(B_r(x_0))}\,\leq\,2\left\|\beta(\cdot,x_0)\right\|_{L^d(B_r(x_0))},
\]
for every $x_0\in B_1$ and $r>0$ with $B_r(x_0)\subset B_1$.
\end{Proposition}
\begin{proof}
Notice that 
\[
	\begin{split}
		&\int_{B_1}\int_{\mathbb{R}^{d^*}}\frac{\left|F(M-Q,x-y)-F(M-Q,x_0-y)\right|}{\left\|M\right\|\,+\,1}\eta_\varepsilon(Q)\eta_\tau(y){\bf d}Q{\bf d}y\\
			&\quad\leq\,\int_{B_1}\int_{\mathbb{R}^{d^*}}\frac{\left\|M-Q\right\|\,+\,1}{\left\|M\right\|\,+\,1}\beta(x-y,x_0-y)\eta_\varepsilon(Q)\eta_\tau(y){\bf d}Q{\bf d}y\\
			&\quad=\,\beta^\tau(x,x_0)\int_{\mathbb{R}^{d^*}}\frac{\left\|M-Q\right\|\,+\,1}{\left\|M\right\|\,+\,1}\eta_\varepsilon(Q){\bf d}Q,
	\end{split}
\]
where $\beta^\tau$ denotes the convolution of $\beta$ with $\eta_\tau$. Because
\[
	\int_{\mathbb{R}^{d^*}}\frac{\left\|M-Q\right\|\,+\,1}{\left\|M\right\|\,+\,1}\eta_\varepsilon(Q){\bf d}Q\,\longrightarrow\,1,
\]
there exists $\varepsilon^*>0$ such that, if $\varepsilon<\varepsilon^*$, we have.
\[
	\beta_{\varepsilon,\tau}(x,x_0)\,\leq 2\beta^\tau(x,x_0)
\]
The H\"older inequality for convolutions completes the proof.
\end{proof}

Because of A\ref{assump_Fosc}, we conclude 
\[
	\left\|\beta_{\varepsilon,\tau}(\cdot,x_0)\right\|_{L^d(B_r(x_0)}^d\,\leq\,(2\beta_0)^dr^d.
\]

In what follows we put forward the proofs of Theorems \ref{teo_main1} and \ref{teo_main2}, accounting for the fixed-coefficients setting.

\section{Modes of convergence}\label{sec_proof12}

In this section we detail the proof of Theorem \ref{teo_main1} and Theorem \ref{teo_main2}. 

\begin{proof}[Proof of Theorem \ref{teo_main1}]
Consider the sequence of operators $(F_n)_{n\in\mathbb{N}}$ defined as
\[
	F_n(M)\,=\,(F\,\ast\,\eta_{\frac{1}{n}})(M),
\]
for every $n\in\mathbb{N}$. Because of Proposition \ref{prop_elliptic} we infer that $F_n$ is a $(\lambda,\Lambda)$-elliptic operator, for every $n\in\mathbb{N}$. Consider the sequence $(u_n)_{n\in\mathbb{N}}\subset\mathcal{C}(B_1)$ of viscosity solutions to
\[
	F_n(D^2u_n)\,=\,0\;\;\;\;\;\mbox{in}\;\;\;\;\;B_1.
\]
Due to the $\mathcal{C}^{1,\beta_0}$-regularity theory available for $F_n=0$, we conclude the family $(u_n)_{n\in\mathbb{N}}$ is equibounded in $\mathcal{C}^{1,\beta_0}_{loc}(B_1)$, for some universal constant $\beta_0\in(0,1)$ not depending on $n\in\mathbb{N}$. Therefore, up to a subsequence if necessary, we have
\begin{equation}\label{eq_ksconv}
	u_n\,\longrightarrow u\;\;\;\;\;\;\;\;\;\;\mbox{in}\;\;\;\;\;\;\;\;\;\;\mathcal{C}^{1,\beta}_{loc}(B_1)
\end{equation}
for every $\beta\in(0,\beta_0)$, where $u\in\mathcal{C}(B_1)$ solves
\[
	F(D^2u)\,=\,0\;\;\;\;\;\mbox{in}\;\;\;\;\;B_1
\]
in the viscosity sense.

Because $F_n$ is a smooth operator, we resort to the partial regularity result to conclude the existence of $\delta>0$ and $\Omega_n\subset B_1$ such that
\[
	u_n\,\in\,\mathcal{C}^{2,\alpha}(B_1\setminus\Omega_n)\;\;\;\;\;\;\;\;\;\;\mbox{and}\;\;\;\;\;\;\;\;\;\;\dim_{\mathcal{H}}\Omega_n\,<\,d\,-\,\delta.
\]
By standard properties of the Hausdorff measure, we have
\[
	\dim_{\mathcal{H}}\left[\bigcup\limits_{n\in\mathbb{N}}\Omega_n\right]\,<\,d\,-\,\delta.
\]
Hence, 
\[
	u_n\,\in\,\mathcal{C}^{2,\alpha}\left(B_1\setminus\bigcup\limits_{n\in\mathbb{N}}\Omega_n\right)
\]
for every $n\in\mathbb{N}$. Together with the convergence in \eqref{eq_ksconv}, the former computation completes the proof.
\end{proof}

In the sequel we prove Theorem \ref{teo_main2}.

\begin{proof}[Proof of Theorem \ref{teo_main2}]
As before, we notice that $F_n$ is of class $\mathcal{C}^\infty$, for every $n\in\mathbb{N}$. If we set $(u_n)_{n\in\mathbb{N}}$ to be the sequence of viscosity solutions to
\[
	F(D^2u)\,=\,0\;\;\;\;\;\mbox{in}\;\;\;\;\;B_1,
\]
there exists $\Omega\subset B_1$ such that
\[
	(u_n)_{n\in\mathbb{N}}\,\subset\,\mathcal{C}^{2,\alpha}(B_1\setminus\Omega),
\] 
for every $\alpha\in(0,1)$, where 
\[
	\dim_{\mathcal{H}}\Omega\,<\,d\,-\,\delta,
\]
for some $\delta>0$.


By combining Proposition \ref{prop_c1loglip} with Remark \ref{rem_rate} we obtain the existence of $C>0$, not depending on $n$, so that
\[
	\sup_{x\in B_r(x_0)}\left|u_n(x)\,-\,u_n(x_0)\,-\,Du_n(x_0)\cdot(x\,-\,x_0)\right|\,\leq\,Cr^2\ln r^{-1},
\]
for every $x_0\in B_{1/2}$ and $0<r<1/2$. Therefore, 
$u_n\longrightarrow u$ locally in $\mathcal{C}^{1,\llip}_{loc}(B_1)$, through a subsequence if necessary. This ends the proof.
\end{proof}

\begin{Remark}[Density in the weak-$W^{2,p}$ topology]\normalfont
In \cite{PimTei16} the authors work under A\ref{assump_Felliptic}-A\ref{assump_recconvex} and prove that solutions to 
\[
	F(D^2u)\,=\,f\;\;\;\;\;\mbox{in}\;\;\;\;\;B_1
\]
are in $W^{2,p}_{loc}(B_1)$. In addition, estimates are available.

The proof of Theorem \ref{teo_main2} implies the approximating sequence $(u_n)_{n\in\mathbb{N}}$ is uniformly bounded in $W^{2,p}_{loc}(B_1)$ for every $p>d$. Therefore,
\[
	u_n\,\rightharpoonup\, u\;\;\;\;\;\mbox{in}\;\;\;\;\;W^{2,p}_{loc}(B_1),
\]
for every $p>d$.
\end{Remark}

The next section accounts for the case of inhomogeneous equations, driven by elliptic operators in the presence of variable coefficients.

\section{Nonhomogeneous problems with variable coefficients}\label{sec_generalizations}

In this section we consider \eqref{eq_main2} and require A\ref{assump_Felliptic}, A\ref{assump_Fosc} and A\ref{assump_source} to hold true. Here, the operator depends explicitly on the space variable. Moreover, equation has a source term $f$. Therefore, further conditions must be imposed on the problem for the partial regularity result to hold.

As mentioned before, in \cite{dosTei16} the authors prove that flat solutions to \eqref{eq_main2} are of class $\mathcal{C}^{1,\llip}$ in general. Under H\"older-continuity conditions of $F$ and $f$ with respect to the spatial variable, this result is improved. In this case, flat solutions are in $\mathcal{C}^{2,\alpha}_{loc}(B_1)$, with the usual estimates. As a consequence, the authors generalize the partial regularity result to the context of \eqref{eq_main2}. Their requirement is Lipschitz-continuity of $F$ and $f$ with respect to $x$; see \cite[Corollary 5.2]{dosTei16}.

In what follows we put forward the proof of Theorem \ref{teo_main3}.
	
\begin{proof}[Proof of Theorem \ref{teo_main3}]
	Let us consider the sequence of operators $(F_{n,m})_{n,m\in\mathbb{N}}$ defined in Section \ref{sec_potc} as
	\[
		F_{n,m}(M,x)\,:=\,\int_{B_1}\int_{\mathbb{R}^{d^*}}F(M\,-\,Q,x\,-\,y)\eta_{\frac{1}{n}}(Q)\eta_{\frac{1}{m}}(y){\bf d}Q{\bf d}y.
	\]
	Moreover, consider the sequence of functions $(f_m)_{m\in\mathbb{N}}$ given by 
	\[
	f_m\,:=\,f\,\ast\, \eta_{\frac{1}{m}}\,=\,\int_{B_1} f(x-y)\eta_{\frac{1}{m}}(y){\bf d}y
	\]
	for every $m\in\mathbb{N}$.
	
	We already know that $F_{n,m}$ is a $(\lambda, \Lambda)$-elliptic operator. Also, because of Proposition \ref{prop_oscreg}, the oscillation $\beta_{n.m}$ is controlled by the oscillation of the operator $F(M,x)$. Let $(u_{n,m})_{n,m\in\mathbb{N}}\subset\mathcal{C}(B_1)$ be a sequence of viscosity solutions to 
	\[
	F_{n,m}(D^2(u_{n,m}),x)=f_m(x) \;\;\;\;\;\mbox{in} \;\;\;\;\; B_1.
	\]  
	
	In what follows, we re-enumerate the sequences $(u_{n,m})_{n,m\in\mathbb{N}}$ and $(F_{n,m})_{n,m\in\mathbb{N}}$ as to write
	\[
	F_{j}(D^2u_j,x)\,=\,f_j(x) \;\;\;\;\; \mbox{in} \;\;\;\;\; B_1
	\] 
	where $F_j(M,x):=F_{j,j}(M,x)$ and $u_j=u_{n,m}$.
	
	By standard regularity results, we have $(u_j)_{j\in\mathbb{N}}\subset \mathcal{C}^{1,\beta_0}_{loc}(B_1)$, with suitable estimates; see \cite{SilTei15}. Hence, through a subsequence if necessary, we have 
	\[
	u_j \longrightarrow u \;\;\;\;\; \mbox{in} \;\;\;\;\; \mathcal{C}^{1,\beta}_{loc}(B_1),
	\]
	for every $\beta\in(0,\beta_0)$.	
	Notice $F_j$ and $f_j$ are smooth in their respective domains. Hence, for $\alpha\in(0,1)$, Proposition \ref{prop_prrvc} yields $\delta>0$ and $\Omega_j\subset B_1$ such that
	\[
	 u_j\,\in\,\mathcal{C}^{2,\alpha}(B_1\setminus\Omega_j)\;\;\;\;\;\;\;\;\;\;\mbox{and}\;\;\;\;\;\;\;\;\;\;\dim_{\mathcal{H}}\Omega_j\,<\,d\,-\,\delta.
	\]
	As in Section \ref{sec_proof12}, we set by $\Omega\,:=\, \bigcup\limits_{j\in\mathbb{N}}\Omega_j$. Properties of the Hausdorff measure lead to
		\[
		\dim_{\mathcal{H}}\Omega\,<\,d\,-\,\delta.
		\]
		and
		\[
		u_j\,\in\,\mathcal{C}^{2,\alpha}\left(B_1\setminus\Omega\right)
		\]
		for every $j\in\mathbb{N}$, which completes the argument.
\end{proof}

\begin{Remark}\normalfont
If we impose $F(\,\cdot\,,x_0)$ to be convex, Theorem \ref{teo_main3} can be refined. In this case convergence takes place in the $\mathcal{C}^{1,\llip}$-topology. As before, weak convergence in $W^{2,p}$ also is found.
\end{Remark}

\medskip 

\noindent{\bf Acknowledgements:} DdP is partially supported by Capes-Fapitec-Brazil and CNPq-Brazil (Grants \#88887.157906/2017-00 and \#422572/2016-0).EP is partially supported by CNPq-Brazil (Grants \#433623/2018-7 and \#307500/2017-9), FAPERJ (Grant \#E.200.021-2018) and Instituto Serrapilheira (Grant \# 1811-25904). GR is partially supported by CAPES. This study was financed in part by the Coordena\c{c}\~ao de Aperfei\c{c}oamento de Pessoal de N\'ivel Superior - Brazil (CAPES) - Finance Code 001.

\medskip

\bibliography{biblio}

\begin{thebibliography}{10}

\bibitem{Armstrong-Silvestre-Smart_2012}
S.~N. Armstrong, L.~Silvestre, and C.~K. Smart.
\newblock Partial regularity of solutions of fully nonlinear, uniformly
  elliptic equations.
\newblock {\em Comm. Pure Appl. Math.}, 65(8):1169--1184, 2012.

\bibitem{Barron-Evans-Jensen_1984}
E.~N. Barron, L.~C. Evans, and R.~Jensen.
\newblock Viscosity solutions of {I}saacs' equations and differential games
  with {L}ipschitz controls.
\newblock {\em J. Differential Equations}, 53(2):213--233, 1984.

\bibitem{Cabre-Caffarelli_2003}
X.~Cabr\'{e} and L.~A. Caffarelli.
\newblock Interior {$C^{2,\alpha}$} regularity theory for a class of nonconvex
  fully nonlinear elliptic equations.
\newblock {\em J. Math. Pures Appl. (9)}, 82(5):573--612, 2003.

\bibitem{Caf89}
L.~A. Caffarelli.
\newblock Interior a priori estimates for solutions of fully nonlinear
  equations.
\newblock {\em Ann. of Math. (2)}, 130(1):189--213, 1989.

\bibitem{CafCab95}
L.~A. Caffarelli and X.~Cabr\'e.
\newblock {\em Fully nonlinear elliptic equations}, volume~43 of {\em American
  Mathematical Society Colloquium Publications}.
\newblock American Mathematical Society, Providence, RI, 1995.

\bibitem{CafSil10}
L.~A. Caffarelli and L.~Silvestre.
\newblock Smooth approximations of solutions to nonconvex fully nonlinear
  elliptic equations.
\newblock In {\em Nonlinear partial differential equations and related topics},
  volume 229 of {\em Amer. Math. Soc. Transl. Ser. 2}, pages 67--85. Amer.
  Math. Soc., Providence, RI, 2010.

\bibitem{dosTei16}
D.~dos Prazeres and E.~V. Teixeira.
\newblock Asymptotics and regularity of flat solutions to fully nonlinear
  elliptic problems.
\newblock {\em Ann. Sc. Norm. Super. Pisa Cl. Sci. (5)}, 15:485--500, 2016.

\bibitem{Escauriaza_1993}
L.~Escauriaza.
\newblock {$W^{2,n}$} a priori estimates for solutions to fully nonlinear
  equations.
\newblock {\em Indiana Univ. Math. J.}, 42(2):413--423, 1993.

\bibitem{Evans_1982}
L.~C. Evans.
\newblock Classical solutions of fully nonlinear, convex, second-order elliptic
  equations.
\newblock {\em Comm. Pure Appl. Math.}, 35(3):333--363, 1982.

\bibitem{Evans-Souganidis_1984}
L.~C. Evans and P.~E. Souganidis.
\newblock Differential games and representation formulas for solutions of
  {H}amilton-{J}acobi-{I}saacs equations.
\newblock {\em Indiana Univ. Math. J.}, 33(5):773--797, 1984.

\bibitem{Fleming-Souganidis_1989}
W.~H. Fleming and P.~E. Souganidis.
\newblock On the existence of value functions of two-player, zero-sum
  stochastic differential games.
\newblock {\em Indiana Univ. Math. J.}, 38(2):293--314, 1989.

\bibitem{Lin_1986}
R.~Hardt, D.~Kinderlehrer, and F.-H. Lin.
\newblock Existence and partial regularity of static liquid crystal
  configurations.
\newblock {\em Comm. Math. Phys.}, 105(4):547--570, 1986.

\bibitem{Krylov_1982}
N.~V. Krylov.
\newblock Boundedly inhomogeneous elliptic and parabolic equations.
\newblock {\em Izv. Akad. Nauk SSSR Ser. Mat.}, 46(3):487--523, 670, 1982.

\bibitem{Krylov-Safonov_1980}
N.~V. Krylov and M.~V. Safonov.
\newblock A property of the solutions of parabolic equations with measurable
  coefficients.
\newblock {\em Izv. Akad. Nauk SSSR Ser. Mat.}, 44(1):161--175, 239, 1980.

\bibitem{Leo09}
G.~Leoni.
\newblock {\em A first course in {S}obolev spaces}, volume 105 of {\em Graduate
  Studies in Mathematics}.
\newblock American Mathematical Society, Providence, RI, 2009.

\bibitem{Nadirashvili-Vladut_2007}
N.~Nadirashvili and S.~Vl\u{a}du\c{t}.
\newblock Nonclassical solutions of fully nonlinear elliptic equations.
\newblock {\em Geom. Funct. Anal.}, 17(4):1283--1296, 2007.

\bibitem{Nadirashvili-Vladut_2008}
N.~Nadirashvili and S.~Vl\u{a}du\c{t}.
\newblock Singular viscosity solutions to fully nonlinear elliptic equations.
\newblock {\em J. Math. Pures Appl. (9)}, 89(2):107--113, 2008.

\bibitem{Nadirashvili-Vladut_2011a}
N.~Nadirashvili and S.~Vl\u{a}du\c{t}.
\newblock Octonions and singular solutions of {H}essian elliptic equations.
\newblock {\em Geom. Funct. Anal.}, 21(2):483--498, 2011.

\bibitem{Nadirashvili-Vladut_2011b}
N.~Nadirashvili and S.~Vl\u{a}du\c{t}.
\newblock Singular solutions of {H}essian fully nonlinear elliptic equations.
\newblock {\em Adv. Math.}, 228(3):1718--1741, 2011.

\bibitem{PimSan18}
E.~A. Pimentel and M.~S. Santos.
\newblock Asymptotic methods in regularity theory for nonlinear elliptic
  equations: a survey.
\newblock In {\em P{DE} models for multi-agent phenomena}, volume~28 of {\em
  Springer INdAM Ser.}, pages 167--194. Springer, Cham, 2018.

\bibitem{PimTei16}
E.~A. Pimentel and E.~V. Teixeira.
\newblock Sharp {H}essian integrability estimates for nonlinear elliptic
  equations: an asymptotic approach.
\newblock {\em J. Math. Pures Appl. (9)}, 106(4):744--767, 2016.

\bibitem{Savin_2007}
O.~Savin.
\newblock Small perturbation solutions for elliptic equations.
\newblock {\em Comm. Partial Differential Equations}, 32(4-6):557--578, 2007.

\bibitem{SilTei15}
L.~Silvestre and E.~V. Teixeira.
\newblock Regularity estimates for fully non linear elliptic equations which
  are asymptotically convex.
\newblock In {\em Contributions to nonlinear elliptic equations and systems},
  volume~86 of {\em Progr. Nonlinear Differential Equations Appl.}, pages
  425--438. Birkh\"{a}user/Springer, Cham, 2015.

\bibitem{Swiech_1997}
A.~\'{S}wi\polhk{e}ch.
\newblock {$W^{1,p}$}-interior estimates for solutions of fully nonlinear,
  uniformly elliptic equations.
\newblock {\em Adv. Differential Equations}, 2(6):1005--1027, 1997.

\bibitem{Trudinger_1989}
N.~S. Trudinger.
\newblock On the twice differentiability of viscosity solutions of nonlinear
  elliptic equations.
\newblock {\em Bull. Austral. Math. Soc.}, 39(3):443--447, 1989.

\bibitem{Winter_2009}
N.~Winter.
\newblock {$W^{2,p}$} and {$W^{1,p}$}-estimates at the boundary for solutions
  of fully nonlinear, uniformly elliptic equations.
\newblock {\em Z. Anal. Anwend.}, 28(2):129--164, 2009.

\end{thebibliography}
\bibliographystyle{plain}

\medskip

\noindent\textsc{Disson dos Prazeres}\\
Department of Mathematics\\
UFS\\
49100-000, S\~ao Crist\'ov\~ao-SE, Brazil\\
\noindent\texttt{disson@mat.ufs.br}

\medskip

\noindent\textsc{Edgard A. Pimentel (Corresponding Author)}\\
Department of Mathematics\\
Pontifical Catholic University of Rio de Janeiro -- PUC-Rio\\
22451-900, G\'avea, Rio de Janeiro-RJ, Brazil\\
\noindent\texttt{pimentel@puc-rio.br}

\medskip

\noindent\textsc{Giane Casari Rampasso}\\
Department of Mathematics\\
University of Campinas -- UNICAMP\\
13083-859, Cidade Universit\'aria, Campinas-SP, Brazil\\
\noindent\texttt{girampasso@ime.unicamp.br}
\end{document}